\documentclass[12pt]{amsart}
\usepackage{tikz}
\usetikzlibrary{calc, arrows, decorations.markings} \tikzset{>=latex}
\usepackage{amsmath,amssymb,amsfonts}
\usepackage{amscd, amssymb, latexsym, amsmath, amscd}
\usepackage{amsrefs}
\usepackage[all]{xy}
\usepackage{pb-diagram}
\usepackage{enumerate}
\usepackage{anysize}
\usepackage[sc]{mathpazo} 
\usepackage{enumitem}
\marginsize{2.5cm}{2.5cm}{2.5cm}{2.5cm}
\usepackage{graphics}
\usepackage{color}
\DeclareFontFamily{U}{wncy}{}
    \DeclareFontShape{U}{wncy}{m}{n}{<->wncyr10}{}
    \DeclareSymbolFont{mcy}{U}{wncy}{m}{n}
    \DeclareMathSymbol{\Sh}{\mathord}{mcy}{"58}

\marginsize{3cm}{3cm}{3cm}{3cm}
\input xy
\xyoption{all}
\usepackage{pb-diagram}
\usepackage[all]{xy}
\input xy
\xyoption{all}
\theoremstyle{plain}
\newtheorem{thm}{Theorem}%[section]
\newtheorem{theorem}{Theorem}[section]
\newtheorem{lem}[theorem]{Lemma}
\newtheorem{lemma}{Lemma}
\newtheorem{cor}[theorem]{Corollary}

\newtheorem{prop}[theorem]{Proposition}

\theoremstyle{definition} \theoremstyle{definition}
\newtheorem{remark}{Remark}

\theoremstyle{remark}

\newcommand{\Sc}{\mathcal{S}}

\newcommand{\M}{\mathfrak{R}}

\newcommand{\Hecke}{\mathcal{H}}

\newcommand{\G}{\textsc{\G}}

\newcommand{\p}{\mathfrak{p}}
\newcommand{\m}{\mathfrak{m}}

\newcommand{\U}{\mathcal{U}}

\newcommand{\Z}{\mathbb{Z}}

\newcommand{\C}{\mathbb{C}}

\newcommand{\wM}{\widehat{M}}

\newcommand{\Ann}{{\rm Ann}}

\newcommand{\Tor}{{\rm Tor}}
\newcommand{\Hom}{{\rm Hom}}
\newcommand{\EP}{{\rm EP}}
\newcommand{\Ext}{{\rm Ext}}
\newcommand{\Ind}{{\rm Ind}}
\newcommand{\ind}{{\rm ind}}

\def\G{{\rm G}}

\def\Sp{{\rm Sp}}

\def\St{{\rm St}}

\def\U{{\rm U}}

\def\GL{{\rm GL}}
\def\PGL{{\rm PGL}}

\def\SO{{\rm SO}}
\def\OO{{\rm O}}

\def\O{{\mathcal O}}

\begin{document}

\title[On a duality theorem of Schneider-Stuhler]
{On a duality theorem of Schneider-Stuhler}
\author{Madhav Nori and Dipendra Prasad}
\address{University of Chicago, Chicago, IL 60637, USA.}
\email{nori@math.uchicago.edu}

\address{Tata Institute of Fundamental 
Research, Homi Bhabha Road, Bombay - 400 005, INDIA.}

\email{dprasad@math.tifr.res.in}

\subjclass{Primary 11F70; Secondary 22E55}

\begin{abstract} We extend a  duality theorem of Schneider-Stuhler
  \cite{schneider-stuhler:sheaves} about $\Ext^i[\pi_1,\pi_2]$ proved there
  for smooth representations of a $p$-adic group $G$
  with central characters to all smooth representations
  assuming their result for only irreducible representations by generalities in
  homological algebra.
  \end{abstract}

\maketitle
{\hfill \today}

\vspace{4mm}

\tableofcontents

\section{Introduction} \label{parameter}

Let 
$\underline{G}$ be a connected reductive algebraic group
over $F$, a non-archimedean local field, and
$G=\underline{G}(F)$ the locally compact group of $F$ rational
points of $\underline{G}$. In 
\cite{schneider-stuhler:sheaves},
page 133, Schneider-Stuhler
prove a {\it Duality Theorem} relating
$\Ext^i[\pi_1,\pi_2]$
with certain ${\rm Tor}^{n-i}[D(\pi_1),\pi_2]$
where $\pi_1$ is a smooth representation of $G$ of finite length, 
$\pi_2$ is a general  smooth representation of $G$, and $D(\pi_1)$ denotes the Aubert-Zelevinsky involution of $\pi_1$.

The theorem of Schneider-Stuhler however assumed that $\pi_1,\pi_2$ had a central character, and  $\Ext^i[\pi_1,\pi_2]$ is calculated in the category of smooth representations of $G$ with that central character. In the presence of non-compact center, the category of smooth representations of $G$ cannot be decomposed using central characters, and therefore to  prove analogous results about $\Ext^i[\pi_1,\pi_2]$ where $\pi_1,\pi_2$ are general
smooth representations of $G$, and one of the representations is irreducible, does not seem a consequence of the theorem of Schneider-Stuhler. For some of the applications the second author
had in mind in \cite{EP} dealing with $\Ext^i_{\GL_n(F)}[\pi_1,\pi_2]$ where $\pi_1,\pi_2$ are smooth representations of
$\GL_n(F)$ with $\pi_1$ the restriction to $\GL_n(F)$ of an irreducible smooth representation of $\GL_{n+1}(F)$, it was important not to restrict oneself to smooth representations with a given central character.

In this paper we give a proof of the Schneider-Stuhler duality theorem
without assuming any conditions on central characters. In fact our proof uses 
the Schneider-Stuhler duality theorem in the simplest possible case,
dealing with $\Ext^i[\pi_1,\pi_2]$ where $\pi_1,\pi_2$ are smooth and irreducible
representations of $G$.

The idea behind the paper can be summarized as follows. Let $A$ be a finitely generated
commutative algebra over $\C$, and $B$ an associative but
not necessarily commutative algebra containing $A$ in its center, such that $B$ is  finitely generated as an $A$-module. Suppose $M \rightarrow {\mathcal F}(M)$
is a functor
from the category of $B$-modules to the category of $A$-modules. Then one can understand the functor $M \rightarrow {\mathcal F}(M)$
via the intermediaries of the functors
$M\rightarrow {\mathcal F}(M/\m^k M)$ and $M\rightarrow {\mathcal F}(\widehat{M})$
where $\m$ is a maximal ideal in $A$, and $\widehat{M}=
\displaystyle{\lim_{\leftarrow}\, }(M/\m^\ell M)$. For us, ${\mathcal F}(M) = \Ext_B^i(N,M)$ or $\Ext_B^j(M,N)$ for $N$ a simple $B$-module on which $A$ acts via $A/\m$. That we are in this
algebraic setup for the extension problem in $p$-adic groups is by the work of Bernstein on `Bernstein center'. The paper assumes that we understand the functor $M \rightarrow {\mathcal F}(M)$ for $M$ a simple $B$-module (through the work of Schneider-Stuhler) and proceeds in steps,
first extending their result (now for representations without central characters) to $M$ of finite length as a $B$-module, then for $M$ finitely generated over $B$, and finally for $M$ an arbitrary $B$-module!

\section{Aubert-Zelevinsky involution}
In this section we discuss the Aubert-Zelevinsky involution 
$\pi \rightarrow D(\pi)$ cf. \cite{aubert}, in some detail as it plays a pivotal role in the Schneider-Stuhler theorem.

We continue with $\underline{G}$ a connected reductive algebraic group
over $F$, a non-archimedean local field, and
$G=\underline{G}(F)$ the locally compact group of $F$ rational
points of $\underline{G}$. Let $\pi$ be an irreducible  smooth representation of $G$. 
Associated to $\pi$ is the  Aubert-Zelevinsky involution $D(\pi)$ of $\pi$ 
which is an element of the Grothendieck group of  smooth representations of $G$ of finite length, 
and defined by:
$$D(\pi)= \Ind_{P_\phi}^G (R_{P_\phi} \pi) -\sum_{P1}{\Ind}_{P1}^G(R_{P1} \pi)
+ \sum_{P2}{\Ind}_{P2}^G(R_{P2} \pi) -+...,$$
where $P_\phi$ is a fixed  minimal parabolic in $G$, $P1$ the next larger parabolics in $G$ containing $P_\phi$,  $P2$ next larger
parabolics etc.; $R_{Pi}\pi$ are the normalized Jacquet
modules of the representation $\pi$ with respect to the parabolic $Pi$, and ${\Ind}$ denotes normalized parabolic induction. More precisely, it has been proved by Deligne-Lusztig for finite
groups of Lie type in \cite{DL}, and Aubert in \cite{aubert} for $p$-adic groups, that 
for $\pi$ an irreducible representation of $G$, the cohomology of the following natural complex (that we will call
Deligne-Lusztig-Aubert complex):
$$0 \rightarrow \pi \rightarrow  \sum_{ |I| = |S|-1}{\Ind}_{P_I}^G(R_{P_I} \pi) \rightarrow \sum_{|I|=|S|-2}{\Ind}_{PI}^G(R_{PI} \pi)
\rightarrow \cdots \rightarrow  \sum_{ |I|=|S|-i}{\Ind}_{PI}^G(R_{PI} \pi)\rightarrow 0 ,$$
is concentrated in top degree, defining $D(\pi)$ up to a sign; here
$S$ denotes the set of simple roots of $G$ with respect to
a  maximal split torus of $G$
contained in  a minimal parabolic $P_\phi$ with
$s=|S|$, and for $I\subset S$, $P_I$ denotes the corresponding parabolic subgroup of $G$ containing $P_\phi$; $i$ denotes the largest integer for which $R_{PI} \pi $ is nonzero for
some $I\subset S$ with $|I|=s-i$.

There is another  involution on the category of finite length
smooth representations of $G(F)$ due to Bernstein in \cite{lectures}. Denote
this other involution as $\pi \rightarrow D'(\pi)$, which  is defined as:
$$D'(\pi) = \Ext^d_{G(F)}[\pi, {\mathcal H}(G(F))],$$
where ${\mathcal H}(G(F))$ is the Hecke algebra of $G(F)$, and where 
$d$ is the only integer
for which $\Ext^d_{G(F)}[\pi, {\mathcal H}(G(F))]$ is nonzero (that there is only one $d$ for an irreducible representation of $G(F)$ is part of \cite{lectures}). 
The Hecke algebra of $G(F)$ being both left and right $G(F)$-module, $\Ext^d_{G(F)}[\pi, {\mathcal H}(G(F))],$ taken for the left $G(F)$ action
on ${\mathcal H}(G(F))$ is a right $G(F)$-module.

It is a consequence of the work of Schneider-Stuhler \cite{schneider-stuhler:sheaves}
that  for an irreducible smooth representation $\pi$ of $G(F)$, $D'(\pi) \cong D(\pi^\vee)$. 
It is known  from \cite{aubert} and \cite{schneider-stuhler:sheaves}
that $D(\pi)$ takes irreducible representations of $G$ to irreducible representations of $G$ (up to  a sign).

In this paper we will simply call $|D(\pi)|$ as the Aubert-Zelevinsky involution,
and  denote it  as $D(\pi)$.

\begin{remark}
  We would like to note two consequences of the Aubert-Zelevinsky involution
  defined through the above Deligne-Lusztig-Aubert complex.
  First, it manifestly implies that if $\pi$ is any irreducible representation of $G(F)$ then $D(\pi)$ is not only a representation in the Grothendieck group of representations of $G(F)$, but an honest representation of $G(F)$ (this is the key for Deligne-Lusztig as well as Aubert for its irreducibility). Second, the definition of $D(\pi)$ via the complex
  makes sense if $\pi$ is any smooth representation of $G(F)$ with all its subquotients having cuspidal supports in the same  Levi subgroup, say $M$ (the complex still has cohomology
  only in the top degree; Aubert's proof in \cite{aubert} never used irreducibility of $\pi$). Since any smooth
  representation of $G(F)$ is a direct sum of representations with
  cuspidal supports in different (standard) Levi subgroups,
  $\pi \rightarrow D(\pi)$ becomes an {\it exact} covariant functor
  from the category of {\it all} smooth representations of $G(F)$ to the category of smooth representations of $G(F)$. It would be interesting to know 
  if the known isomorphism $D'(\pi) \cong D(\pi^\vee)$ for
  $\pi$ an irreducible representation
  of $G(F)$ holds good for $\pi$ any finite length  representation
  of $G(F)$ (better still, a functorial isomorphism), and therefore, if known properties of
  $D'$ on finite length representations due to Bernstein in \cite{lectures} can be
  transported to $D(\pi)$ for finite length representations of $G(F)$; specially, if
  $D(D(\pi))  \cong\pi$, and 
  $D(\Ind_P^G \pi) \cong  \Ind_{P^-}^G D(\pi)$ (where $P^-$ is the parabolic which is
  opposite of P) hold good for for all smooth representations of $G(F)$ of finite length
  (these are usually asserted only up to semi-simplification).
\end{remark}

Here is an example to put the ideas in the previous remark  for use in understanding
the Aubert-Zelevinsky involution in some explicit cases using non-semi-simple representations in an essential way. 

\begin{prop} Let $G=\underline{G}(F)$ be a reductive $p$-adic group with $P=MN$
  a parabolic in $G$.
  Let $V$
be a {\it regular} supercuspidal representation of $M$, i.e., $V^w \not \cong V$ for $w$ any nontrivial element of  
$W_M\backslash W_G/W_M
$ where $W_G$ (resp. $W_M$) is the Weyl group
of $G$ (resp. of $M$) for a maximal split torus in $G$, where $V^w \not \cong V$ means that
either $w$ does not preserve $M$, or if it does, it does not preserve the
isomorphism class of $V$.
Then the principal series representation ${\rm Ind}_P^GV$ 
has a unique irreducible quotient representation $Q(V)$,
and a unique irreducible sub-representation $S(V)$.
The Aubert-Zelevinsky involution of $Q(V)$ is $S(V)$, and that of
$S(V)$ is $Q(V)$. 
\end{prop}
\begin{proof}By the geometric lemma,
  the Jacquet module $R_N({\rm Ind}_P^GV)$
  with respect to $P$ is up to
  semi-simplification, the representation of $M$ 
  $$R_N \left ({\rm Ind}_P^GV \right )
  \cong  \sum_{w \in W_M\backslash W_G/W_M}
  V^w,$$ where the sum is taken over
those elements of the  double coset space
$W_M\backslash W_G/W_M$ which preserve $M$. By hypothesis,
  this sum consists
  of distinct supercuspidal representations of $M$, and therefore, the 
  Jacquet module of ${\rm Ind}_P^GV$ with respect to $P$ is 
  semi-simple, and each component appears with multiplicity 1. By a standard application
  of Frobenius reciprocity, the uniqueness of the  irreducible quotient and
  sub-representation follows.

  Although we wish to calculate the Aubert-Zelevinsky involution $D(\pi)$, not knowing its
  properties on finite length representations, we instead will use the other involution
  $D'(\pi)$, and finally use that $D(\pi) \cong D'(\pi^\vee)$ on irreducible
  representations to conclude the proposition.

  Note that $D'$ is an exact contravariant functor with  (cf.
  Theorem 31 of \cite{lectures}):
  $$D' \left ({\rm Ind}_P^GV \right ) =
   {\rm Ind}_{P^-}^G (V^\vee)  ,$$
 where $P^-=MN^-$ is the opposite parabolic to $P=MN$. Now $S(V)$ is a
 subrepresentation of  ${\rm Ind}_P^GV $, therefore
 $D' \left ({\rm Ind}_P^GV \right )
 = {\rm Ind}_{P^-}^G(V^\vee) $
 has the irreducible
 representation $D'(S(V))$ as a
 quotient. But irreducible quotient of
 ${\rm Ind}_{P^-}^G(V^\vee) $
 being unique, if we can
 prove that $Q(V)^\vee$ is a quotient of ${\rm Ind}_{P^-}^G(V^\vee) $,
 then it will follow that 
 $D'(S(V)) \cong Q(V)^\vee$,
 therefore $D(S(V)) \cong Q(V)$ as desired.

 Finally, to
 prove that $Q(V)^\vee$ is a quotient of ${\rm Ind}_{P^-}^G(V^\vee) $, dualizing, we need
 to prove that $Q(V)$ is a sub-representation of  ${\rm Ind}_{P^-}^G(V) $ given that $Q(V)^\vee$ is a sub-representation  of  ${\rm Ind}_{P}^G(V^\vee) $. This amounts by the Frobenius reciprocity to the well known assertion (applied here to $ Q(V)$) that,
 $$ (\pi^\vee)_{N^-} \cong (\pi_N)^\vee,$$
 which is part of the second adjointness theorem of Bernstein, cf. theorem 21 of \cite{lectures}. 
 \end{proof}

\section{The theorem of Schneider-Stuhler}
Let $G=\underline{G}(F)$ be the locally compact group of $F$ rational
points of a reductive algebraic group $\underline{G}$, and $Z$ its center.
Let $\Hecke(G)$ be the Hecke algebra
of $G$, and for a character $\chi:Z \rightarrow \C^\times$,
let $\Hecke_\chi=\Hecke_\chi(G)$ be the Hecke
algebra of $\chi$-invariant functions on $G$ (locally constant
with compact support modulo $Z$). Integration along $Z$:
$f(g) \rightarrow \int_Zf(gz)\chi(z)dz$ defines a surjective
algebra homomorphism from $\Hecke(G)$ to $\Hecke_\chi(G)$.
The algebras $\Hecke(G)$ and $\Hecke_\chi(G)$ are algebras without units but with a rich supply of idempotents which allows one
to define {\it non-degenerate} representations of these algebras which have the property that $\Hecke(G)V=V$
(resp. $\Hecke_\chi(G)V=V$). There is the well-known equivalence of the category of smooth representations of $G$ and non-degenerate representations of $\Hecke(G)$, and similarly
the category of smooth representations of $G$ with central character $\chi$ and non-degenerate representations of $\Hecke_\chi(G)$.

Let $\M(G)$ be the abelian category of smooth representations of $G$,
and for a character $\chi:Z \rightarrow \C^\times$, let
$\M(G;\chi)$ be the abelian sub-category of smooth representations of $G$ on which $Z$ operates by $\chi$. We use $\Ext^i_{G}[V,V']$
to denote Ext groups in $\M(G)$,
and $\Ext^i_{G,\chi}[V,V']$ to denote Ext groups in $\M(G;\chi)$. Similarly,
$\Tor^{{\mathcal H}}_{i}[V, V'],$
or $\Tor^{G}_{i}[V, V'],$ will denote Tor groups for $\Hecke(G)$-modules, and 
$\Tor^{{\mathcal H}_\chi}_{i}[V, V'],$ or $\Tor^{G,\chi}_{i}[V, V'],$
will denote Tor groups for $\Hecke_\chi$-modules.

Here is the theorem of Schneider-Stuhler \cite{schneider-stuhler:sheaves}, page 133.

\begin{thm}
  Let $G$ be a reductive $p$-adic group with $Z$ as its center, and $\chi:Z\rightarrow \C^\times$ a character. Let $V \in \M(G,\chi)$ be
  an irreducible, admissible representation of $G$ appearing as a sub-quotient of a principal series representation
  $\Ind_P^G \rho$ for a cuspidal representation $\rho$ of a Levi subgroup $M$ of $P$ with $d_s=d_s(V)$ = rank of the maximal split torus in $Z(M) \cap [G,G]$.
  Then for any smooth representation $V' \in \M(G,\chi)$ of $G$, there is a
natural isomorphism
$$\Ext^i_{G,\chi}[V,V']
\cong \Tor^{{\mathcal H}_\chi}_{d_s-i}[D(V^\vee), V'],$$
given by the cap product with $\Tor^{{\mathcal H}_\chi}_{d_s}[D(V^\vee), V] \cong \C$.
  \end{thm}

The following well-known lemma  converts $\Tor$ into an $\Ext$ eliminating the need of $\Tor$ in the above theorem which may be more useful in some contexts (such as in `branching laws'). 

\begin{lem}\label{lemma2}
(a)  For any two smooth 
representations $\pi_1,\pi_2$ of a reductive $p$-adic group $G$, there is a 
canonical isomorphism,
$$\Ext^i_G[\pi_1,\pi^\vee_2] \cong \Ext^i_G[\pi_2,\pi^\vee_1] \cong
\Hom_\C[\Tor^G_i[\pi_1,\pi_2], \C] .$$

(b) Similarly, for any two smooth 
representations $\pi_1,\pi_2$ of a reductive $p$-adic group $G$ with a given central character $\chi: Z \rightarrow \C^\times$, there is a 
canonical isomorphism,
$$\Ext^i_{G,\chi}[\pi_1,\pi^\vee_2] \cong \Hom_\C[\Tor^{G,\chi}_i[\pi_1,\pi_2], \C] \cong \Hom_\C[\Tor^{G,\chi}_i[\pi_2,\pi_1], \C].$$
\end{lem}

\begin{proof} Let,
$$\cdots \rightarrow P_2 \rightarrow P_1 \rightarrow P_0 \rightarrow \pi_2 \rightarrow 0$$
be a projective resolution of $\pi_2$ in $\M(G)$.
By taking the contragredient, we have an injective resolution of $\pi_2^\vee:$
$$0 \rightarrow \pi_2^\vee \rightarrow P_0^\vee \rightarrow P_1^\vee \rightarrow P_2^\vee \rightarrow \cdots$$

Applying the functor $\Hom_G[\pi_1,-]$ to this exact sequence, and omitting the first term, we get the
cochain complex:
$$\Hom[\pi_1,P_\star] = \left \{ 0 \rightarrow \Hom_G[\pi_1, P_0^\vee]  \rightarrow \Hom_G[\pi_1, P_1^\vee]  \rightarrow \Hom_G[\pi_1, P_2^\vee] \rightarrow \cdots \right \}$$
whose cohomology is by definition $\Ext^i_G[\pi_1, \pi_2^\vee]$.

For any smooth representation $V$ of $G$,
define $$\pi_1 \otimes_{\mathcal H} V = \frac{\pi_1 \otimes V}{\{v_1 \otimes v - hv_1 \otimes h^\vee v | h \in \Hecke, v_1 \in \pi_1,v\in V\}},$$
where $h \rightarrow h^\vee$ is the anti-automorphism of $\Hecke = \Hecke(G)$ given by $h^\vee(g)=h(g^{-1})$. Clearly, we have,
$$\Hom_\C[\pi_1 \otimes_{\mathcal H} V, \C]
=\Hom_{\Hecke}[\pi_1, V^\star]= \Hom_{\Hecke}[\pi_1, V^\vee]=
\Hom_G[\pi_1, V^\vee],$$
where $V^\star$ is the vector space dual of $V$, and $V^\vee$ is the smooth dual
of $V$ (as a $G$-module).

This allows one to re-write the  cochain complex $\Hom[\pi_1,P_\star]$ as:
$$0 \rightarrow \Hom_\C[\pi_1 \otimes_{\mathcal H} P_0, \C] \rightarrow \Hom_\C[\pi_1 
\otimes_{\mathcal H} P_1, \C]  \rightarrow \Hom_\C[\pi_1 \otimes_{\mathcal H} P_2, \C ] \rightarrow \cdots$$

This complex is just the dual of the complex,
$$\cdots  \rightarrow \pi_1 \otimes_{\mathcal H} P_2 \rightarrow \pi_1 
\otimes_{\mathcal H} P_1 \rightarrow \pi_1 \otimes_{\mathcal H} P_0\rightarrow 0,$$
which calculates $\Tor^G_i[\pi_1,\pi_2]$. Since taking cohomology of a complex over $\C$ commutes with taking duals,
this completes the proof of part $(a)$ of the lemma.

Part $(b)$ of the lemma is similarly proved by replacing $\M(G)$ by $\M(G;\chi)$, and $\Hecke$ by $\Hecke_\chi$.
\end{proof}

We will use theorem 1 due to  Schneider-Stuhler to prove the following theorem which is the main result of this paper.

\begin{thm}
  Let $G$ be a reductive $p$-adic group, and $\pi$ an irreducible, admissible representation of $G$. 
Let $d(\pi) $ be the   largest integer $i \geq 0$  such that there is an irreducible, admissible representation 
 $ \pi'$  of $G$  with ​
$\Ext^i_G[\pi,\pi']$  
nonzero​.  Then,

\begin{enumerate}
\item There is a unique irreducible representation $\pi'$ of $G$ with $\Ext^{d(\pi)}_G[\pi,\pi'] \not = 0.$   
\item The representation $\pi'$ in $(1)$ 
is nothing but $D(\pi)$  where $D(\pi)$ is the Aubert-Zelevinsky involution of $\pi$, and $d(\pi)$ is the split rank of the Levi subgroup $M$ of $G$ which carries the cuspidal support of $\pi$. 
\item $\Ext_G^{d(\pi)}[\pi, D(\pi)] \cong \C$. 
\item For any smooth  representation $\pi'$ of $G$, the bilinear pairing
$$(*) \,\,\,\,\,   \Ext^{i}_G[\pi, \pi']  \times \Ext^{j}_G[\pi', D(\pi)]  \rightarrow 
\Ext^{i+j = d(\pi)}_G[\pi, D(\pi)]   \cong \C, $$
is nondegenerate in the sense that if $\pi' = \displaystyle{\lim_{\rightarrow}\, } \pi'_n$ of
finitely generated $G$-sub-modules $\pi'_n$,
then $\Ext^{i}_G[\pi, \pi'] =\displaystyle{\lim_{\rightarrow}\, }
\Ext^{i}_G[\pi, \pi_n']$,
a direct limit of finite dimensional vector spaces over $\C$, and 
$\Ext^{j}_G[\pi', D(\pi)]  = \displaystyle{\lim_{\leftarrow}\, } \Ext^{j}_G[\pi'_n, D(\pi)]  $, an inverse limit
of finite dimensional vector spaces over $\C$, and the pairing  in $(*)$ is the direct limit of
perfect pairings on these finite dimensional spaces:
$$\Ext^{i}_G[\pi, \pi_n']  \times \Ext^{j}_G[\pi'_n, D(\pi)]  \rightarrow 
\Ext^{i+j = d(\pi)}_G[\pi, D(\pi)]   \cong \C.$$
(Observe that a compatible family of perfect pairings on finite dimensional vector spaces
$B_n: V_n \times W_n \rightarrow \C$ with $V_n$ part of an inductive system, and $W_n$ part of a projective system, gives rise to a natural pairing $B :
\displaystyle{\lim_{\rightarrow}\, }V_n \times \displaystyle{\lim_{\leftarrow}\, } W_n \rightarrow \C$
such that the associated homomorphism
from $(\displaystyle{\lim_{\rightarrow}\, }V_n)^\star$ to  $\displaystyle{\lim_{\leftarrow}\, } W_n$ is an isomorphism.
\end{enumerate}
\end{thm}

The proof of this theorem will be achieved in two steps. We will take the first step in this section
proving the theorem assuming $\pi'$ to be an irreducible representation of $G$ by a minor modification
of the result of Schneider-Stuhler. Having proved the theorem for $\pi'$ irreducible, the rest of the paper will prove part $(4)$ of the theorem
for a general smooth representation $\pi'$ of $G$.

  For $\pi'$ an irreducible representation of $G$,
  we note that if this theorem is true for reductive groups $G=G_1$ and $G=G_2$,
  then it is true for $G=G_1 \times G_2$ by the Kunneth theorem. We will not need to use Kunneth theorem in this paper but for a proof, see    \cite{raghuram}.

  Also, this theorem is clearly true for tori $T$ (where
  the Aubert-Zelevinsky involution is trivial) on noting that $\Ext^i_T[\chi_1,\chi_2] = 0$
  for all $i$ for any two characters $\chi_1,\chi_2:T\rightarrow \C^\times$,
  if $\chi_1 \not = \chi_2$, and that
  $$\Ext^i_T[\chi,\chi] \cong \Ext^i_T[\C,\C]
  = \Lambda^i(\C^d),$$
where $T=T^c \times \Z^d$ with $T^c$, the maximal compact subgroup of $T$.

Let $D          {G} = [\underline{G},\underline{G}]$, the derived subgroup of $\underline{G}$, and $Z$ the center of ${G}$. By the
above remarks, the theorem  is true for the group $G'= DG(F) \times Z$, and
$\pi'$ an irreducible representation of $G'$. Since $G'$ is a normal subgroup of $G$ of finite index, it may
appear that the generality,
  $$\Ext^i_{G}[\pi_1, \pi_2] =  \Ext^i_{G'}[ \pi_1, \pi_2]^{G},$$
will prove the theorem for irreducible representation $\pi$ of $G$ knowing it for irreducible representations $\pi'$ of $G'$. This does not seem to be the case since an irreducible representation of $G$ may decompose when restricted to $G'$, and therefore even to conclude
$$ \Ext_G^{d(\pi)}[\pi, D(\pi)] \cong \C,$$
seems not obvious. (A case in point would be when an irreducible representation $\pi$ of $G$ decomposes as $\pi = 2 \pi'$ when restricted to $G'$. From the information that $\Ext^d_{G'}[\pi',D\pi'] = \C$, and therefore $\Ext^d_{G'}[\pi,D\pi] = \C^4$,
we need to deduce that $\Ext^d_{G}[\pi,D\pi] = \C$, which seems not clear.)

If the above mentioned obstacle to deducing theorem 2 was not there, we would be using
the theorem of Schneider-Stuhler only for semi-simple groups (and only for irreducible representations of them), that of course would have been preferable, but not having succeeded in that direction, we
will use 
the theorem of Schneider-Stuhler for reductive groups (but only for irreducible representations of them) in a slightly different approach.

\begin{prop}
  Fix a surjective map $\phi: G\rightarrow \Z^d$
  whose kernel $G_\phi$ is an open subgroup of $G$, such that
  $G_\phi$ contains  $G^c = (DG)(F) \cdot Z^c(F)$ as a subgroup of finite index, where $Z^c(F)$ is the maximal
  compact subgroup of $Z$. Then if $\pi$ is a finitely generated
  smooth representation of $G$
  with central character
  $\chi: Z \rightarrow \C^\times$,
  and is
  a projective module in $\M(G,\chi)$, 
  $\pi|_{G_\phi}$ is a projective module in 
  $\M(G_\phi)$.

(b)  For the representation $Q_0 = \Sc(\Z^d)= \ind_{\langle e \rangle }^{\Z^d}(\C)
  $ of $\Z^d$, treated as a representation of $G$ via the map $\phi: G\rightarrow \Z^d$, and $\pi$ a smooth representation of $G$ which is 
a projective module in $\M(G,\chi)$,  $\pi \otimes Q_0$
is a smooth representation of $G$ which is a
projective module in $\M(G)$.
  \end{prop}
\begin{proof} Since $\pi$ is a finitely generated $G$-module 
with central character
$\chi: Z \rightarrow \C^\times$, and $Z\cdot G_\phi$ is of finite index in $G$, $\pi$ is finitely generated as a $G_\phi$ module.

Since $Z \cap G_\phi$ is a compact abelian group contained in the center of
$G_\phi$, it decomposes 
any smooth representation  of $G_\phi$ into a direct sum of eigenspaces for
$Z\cap G_\phi$:
$$V = \sum_\alpha V_{\alpha},$$
where $V_\alpha$ is the subspace of $V$ on which $Z\cap G_\phi$ acts by the character $\alpha :Z\cap G_\phi \rightarrow \C^\times$.

Since $\pi$ is finitely generated as a $G_\phi$ module, we have:
$$\Hom_{G_\phi}[\pi,V] = \sum_\alpha \Hom_{G_\phi}[\pi,V_{\alpha}].$$

Therefore to prove the projectivity of $\pi$ as a $G_\phi$-module,
it suffices to consider only those surjective homomorphisms
$$\lambda:V_1 \rightarrow V_2 \rightarrow 0,$$
of $G_\phi$-modules on which
$Z \cap G_\phi$ acts by the restriction
of $\chi$ to $Z \cap G_\phi$.

Since $Z \cap G_\phi$ acts by a character which is $\chi|_{Z \cap G_\phi}$
  on both $V_1$ and $V_2$, we can let $Z$ operate
on $V_1$ and $V_2$ by $\chi$, giving a structure 
of $Z \cdot G_\phi$
module to both $V_1$ and $V_2$, making
$\lambda:V_1 \rightarrow V_2 $,  
$Z \cdot G_\phi$-equivariant. By inducing these representations to $G$, we get:
$$\ind(\lambda):\ind_{Z\cdot G_\phi}^G (V_1)
\rightarrow \ind_{Z\cdot G_\phi}^G (V_2) \rightarrow 0.$$

Observe that by Frobenius reciprocity,
$$\Hom_G[\pi, \ind_{Z\cdot G_\phi}^G (V_i)] = \Hom_{Z\cdot G_\phi}[\pi, V_i],$$
for both $i=1,2$. But $Z$ operates on $V_i$ as well as $\pi$ by $\chi$, hence,
$$\Hom_G[\pi, \ind_{Z\cdot G_\phi}^G (V_i)]
= \Hom_{Z\cdot G_\phi}[\pi, V_i]
= \Hom_{ G_\phi}[\pi, V_i].$$

Therefore an element  $\phi_2 \in \Hom_{ G_\phi}[\pi, V_2]$
can be interpreted as an element $\varphi_2 \in \Hom_G[\pi, \ind_{Z\cdot G_\phi}^G (V_2)]$.
Since $\pi$ is a projective module in $\M(G,\chi)$,
and both the representations
$\ind_{Z\cdot G_\phi}^G (V_1)$ and $\ind_{Z\cdot G_\phi}^G (V_2)$ are in 
$\M(G,\chi)$ with a surjection, $\ind(\lambda):\ind_{Z\cdot G_\phi}^G (V_1)
\rightarrow \ind_{Z\cdot G_\phi}^G (V_2) \rightarrow 0,$ $\varphi_2$ can be lifted to $\varphi_1$
in $\Hom_G[\pi, \ind_{Z\cdot G_\phi}^G (V_1)]$, and hence $\phi_2
\in \Hom_{ G_\phi}[\pi, V_2]$
can be lifted  to $\phi_1$ in 
$\Hom_{ G_\phi}[\pi, V_1]$, proving the projectivity of  $\pi$ in $\M(G_\phi)$.

For part $(b)$ of the proposition, note that 
$$\pi \otimes Q_0  = \pi \otimes \ind_{G_\phi}^{G}(\C)
= \ind_{G_\phi}^{G}(\pi|_{G_\phi}). $$

By first part of the proposition, $\pi|_{G_\phi}$ is a projective $G_\phi$-module, hence
the following 
most primitive form of the Frobenius reciprocity in the next Lemma completes the proof of
projectivity of $\pi \otimes Q_0 = \ind_{G_\phi}^{G}(\pi|_{G_\phi})$ as a $G$-module.
\end{proof}

\begin{lem}Let $G_\phi$ be an open subgroup of a $p$-adic group $G$. Let $E$ be a smooth representation of 
$G_\phi$, and $F$ a smooth representation of $G$. Then,
$$\Hom_G[ {\rm ind}_{G_\phi}^G E, F] \cong \Hom_{G_\phi} [E, F|_{G_\phi}].$$
\end{lem}

\begin{prop}\label{prop3.2}
  Let $\pi_1$ and $\pi_2$ be two smooth irreducible representations of $G$ with the same central character $\chi: Z\rightarrow \C^\times$. Let
  $\phi:G\rightarrow \Z^d$
  be a surjective homomorphism with kernel $G_\phi$ as before with $G_\phi\cap Z$, the maximal compact subgroup of $Z$. Then
  $$\Ext^k_{G}[\pi_1,\pi_2] \cong \sum_{k=i+j}
  \Ext^i_{G,\chi}[\pi_1,\pi_2] \otimes \Ext^j_{\Z^d}[\C,\C],$$
  where $\Ext^j_{\Z^d}[\C,\C]$ denotes extension of trivial modules of $\Z^d$.
  \end{prop}
\begin{proof}
Let
$$P_* = \cdots \rightarrow P_1 \rightarrow P_0 \rightarrow \pi_1 \rightarrow 0,$$ 
be a projective resolution for $\pi_1$ in $\M(G,\chi)$,
and let $$Q_*= \cdots \rightarrow Q_1 \rightarrow Q_0 \rightarrow \C \rightarrow 0,$$
be a projective resolution
for the trivial module $\C$ for the group $\Z^d$, or for the corresponding $A$-module
for $$A= 
\C[z_1,\cdots,z_d,z_1^{-1},\cdots z_d^{-1} ] =\C[\Z^d] = \Hecke(\Z^d)\cong
\ind_{\langle e \rangle}^{\Z^d} (\C).$$

By the previous proposition,
it follows that the tensor product $P_* \otimes Q_*$: 
$$\cdots \rightarrow P_1\otimes Q_0 + P_0 \otimes Q_1  \rightarrow P_0\otimes Q_0  \rightarrow \pi_1
\rightarrow 0,$$ 
is a projective resolution of $\pi_1$ as a smooth $G$-module, and therefore also a projective resolution of $\pi_1$ as a smooth $G^0$-module where  $G^0 = Z\cdot G_\phi$.
Therefore, 
$\Ext^k_{G^0}[\pi_1, \pi_2]$ is the cohomology of the 
cochain complex  $\Hom_{G^0}[\bigoplus_{i+j = k} P_i\otimes Q_j, \pi_2]$. 

Recall that each $Q_j$ are direct sum of $\Z^d$-modules,
$\Hecke(\Z^d)$,
which considered as a $G$ module via the map $\phi:G\rightarrow \Z^d$
with kernel $G_\phi$, is nothing but 
$\ind_{G_\phi}^{G} (\C)$. As representations of $G^0$, $Q_j=\ind_{G_\phi}^{G^0} (F_j)$ for $F_j$ a finite dimensional vector space over $\C$ on which $G_\phi$ operates trivially.

Towards the calculation of $\Ext^i_{G^0}[\pi_1, \pi_2]$, let us note that
for $F_j$ a finite dimensional vector space over $\C$ on which $G_\phi$ operates trivially,
\begin{eqnarray*}
  \Hom_{G^0}[ P_i\otimes \ind_{G_\phi}^{G^0} (F_j),    \pi_2] & = & 
  \Hom_{G^0}[ \ind_{G_\phi}^{G^0}(P_i|_{G_\phi} \otimes F_j) , \pi_2] \\
  & = & \Hom_{G_\phi}[P_i\otimes F_j , \pi_2]  \\
& = & \Hom_{G_\phi}[P_i , \pi_2] \otimes \Hom_{G_\phi}[F_j,\C] \\
& = & \Hom_{G^0}[P_i , \pi_2] \otimes \Hom_{{\mathcal L}^0} [Q_j,\C],
\end{eqnarray*}  
where ${\mathcal L}^0 = \phi(G^0) \subset \Z^d$, and in the last equality, we have used that $G^0=Z\cdot G_\phi$, and that
$P_i$ and  $\pi_2$ are $G^0$-modules with the same
central character $\chi:Z\rightarrow \C^\times$.

  Summarizing the discussion above, the natural mapping of (tensor product of) the chain complexes:
  $$\Pi: \Hom_{G^0}[P_i, \pi_2 ] \otimes \Hom_{{\mathcal L}^0}[Q_j,\C]
  {\longrightarrow} \Hom_{G^0}[ P_i\otimes Q_j, \pi_2], $$
  is an isomorphism of chain complexes 
    which proves that,
$$\Ext^k_{G^0}[\pi_1,\pi_2] \cong \sum_{k=i+j}
  \Ext^i_{G^0,\chi}[\pi_1,\pi_2 ] \otimes \Ext^j_{{\mathcal L}^0}[\C,\C].$$   

  (A small subtlety in the proof of the proposition lies in the fact that the corresponding homomorphism $\Pi$ of chain complexes with $(G^0,{\mathcal L}^0)$
  replaced by  $(G, \Z^d)$ is not an isomorphism, but still the conclusion about the Ext groups is true.)
  
Observe that all the modules appearing in the isomorphism $\Pi$ above carry
  an action of $G$ (and are projective objects in appropriate categories), and that $\Pi$ is equivariant under the action of $G/G^0$. 
  Since taking cohomology of a complex over $\C$, and
  taking $G/G^0$-invariant, or taking $G/G^0$-invariant of the complex and taking the cohomology is the same, it follows that,
  $$\Ext^k_{G}[\pi_1,\pi_2] \cong \sum_{k=i+j}
  \left [\Ext^i_{G^0,\chi}[\pi_1,\pi_2 ] \otimes  \Ext^j_{{\mathcal L}^0}[\C,\C] \right ]^{G/G^0}.$$
  Now making the crucial observation that the action of $G/G^0$ on
  $\Ext^j_{{\mathcal L}^0}[\C,\C]$ via $\Z^d/{\mathcal L}^0$ is trivial, we deduce
  that:
\begin{eqnarray*}
  \Ext^k_{G}[\pi_1,\pi_2] & \cong & \sum_{k=i+j}
  \left [ \Ext^i_{G^0,\chi}[\pi_1,\pi_2 ] \otimes  \Ext^j_{{\mathcal L}^0}[\C,\C] \right ] ^{G/G^0} \\
  &=& \sum_{k=i+j} \Ext^i_{G^0,\chi}[\pi_1,\pi_2 ]^{G/G^0} \otimes  \Ext^j_{{\mathcal L}^0}[\C,\C] \\
   &=& \sum_{k=i+j} \Ext^i_{G,\chi}[\pi_1,\pi_2 ] \otimes  \Ext^j_{\Z^d}[\C,\C], 
\end{eqnarray*}
where in the last equality we are using that the restriction map
from $\Ext^j_{\Z^d}[\C,\C]$ to $\Ext^j_{{\mathcal L}^0}[\C,\C]$ is an isomorphism,  proving the proposition.
\end{proof}

\begin{remark}
  At this point we have proved theorem 2 for $\pi'$ an irreducible smooth representation of any reductive $p$-adic group $G$ as a consequence of theorem 1 by combining lemma \ref{lemma2}  and Proposition \ref{prop3.2}. The rest of the paper will deduce theorem 2 for {\it any} smooth representation
  $\pi'$ of $G$ from the irreducible case. 
  \end{remark} 

\section{Matlis duality and Injective resolutions}

The results in this paper need  injective resolutions with certain properties
which should certainly be well-known, but not finding a suitable reference, we take the occasion to give  proofs emphasizing that nothing is new.

We begin by recalling the Matlis duality, cf. \cite{Matlis},  first for a commutative algebra $A$, and later
in a non-commutative case needed in this paper. In the commutative case,
Matlis duality is a natural functor from the category of
modules over a (noetherian) local ring $(A,\m)$ to modules over $A$ turning a
noetherian $A$-module  to an artinian $A$-module, and projective $A$-modules to injective $A$-modules.
It is especially easy
to describe for local rings $(A,\m)$ for which $A/\m = k$ is contained in $A$
as is the case in all our applications; we will assume this to be the case, making $A$ a $k$-algebra, and any
$A$-module $M$ as a vector space over $k$. 

Given a module $M$ over a local ring $(A,\m)$,
the Matlis dual of $M$, to be denoted as $M^\vee$, is the
$A$-module $$M^\vee=
\displaystyle{\lim_{\rightarrow}\, }\Hom_k[M/\m^\ell M, k].$$

Note that for an inverse system $\{V_n\}$ of finite dimensional vector spaces $V_n$
over $k$, there is a canonical
isomorphism:
$$  \Hom_k[\displaystyle{\lim_{\rightarrow}\, } \Hom_k[V_n,k], k] \cong
\displaystyle{\lim_{\leftarrow}\, } V_n.$$
It follows that for any finitely generated $A$-module $M$,
there is an isomorphism of $A$-modules,
$$\Hom_k[M^\vee,k] \cong \widehat{M} =\displaystyle{\lim_{\leftarrow}\, } M/\m^\ell M.$$

Since $M\rightarrow \widehat{M}$ is an exact covariant functor from the category
of finitely generated $A$-modules to the category
of $A$-modules,
it follows that
$M\rightarrow M^\vee$ is an exact contravariant functor from the category
of finitely generated $A$-modules to the category
of $A$-modules, taking finitely generated projective $A$-modules to injective $A$-modules,
finitely generated projective $A$-modules to artinian $A$-modules, and
for any $A$-module $M$, $m^\vee$ has the property that
$M^\vee = \cup_{k\geq 1} \Ann(\m^k;M^\vee),$ where for any $A$-module $M$, we let $\Ann(\m^k;M)
= \{m \in M|\m^k m = 0\}$.

    Let $$\rightarrow P_n\rightarrow \cdots \rightarrow P_1\rightarrow P_0\rightarrow N\rightarrow 0,$$
  be a finitely generated
  projective resolution of a finitely generated $A$-module $N$.
  Applying the Matlis duality,
  we get an injective resolution:
$$0 \rightarrow N^\vee \rightarrow  P^\vee_0\rightarrow   P^\vee_1 \rightarrow P^\vee_2 \rightarrow \cdots ,$$
    of $N^\vee$.

    If $N=A/m$, then clearly $N^\vee = A/m$; more generally,
    if $N$ is a finite artinian $A$-module, then $N^\vee$ too is one,
    allowing us to construct injective resolution of artinian $A$-modules. We summarize
    this conclusion in the following proposition.

    \begin{prop} \label{matlis} Let $(A,\m)$ be a local 
      $k=A/\m$ algebra which is finitely generated over $k$.
      Then any artinian $A$-module $N$ killed by a power of $\m$ has an injective resolution:
      $$0 \rightarrow N\rightarrow I_0 \rightarrow  I_1  \rightarrow  I_2 \rightarrow \cdots,$$
      by artinian and injective $A$-modules $I_j$ with $I_j= \cup_{k\geq 1} \Ann(\m^k;I_j)$ for all $j$. 
      \end{prop}

    \begin{remark} Despite simplicity of ideas in this section, one may note that for $A$ considered as a module over itself, the $A$-module
      $A^\vee = \displaystyle{\lim_{\rightarrow}\, }\Hom_k[A/\m^\ell, k]$
      which is nothing but the injective hull of $k = A/\m$, is not so easy to describe explicitly.
      \end{remark}

    \begin{remark}The Matlis dual defined here is slightly different in appearance from the paper of Matlis who defines it as $\Hom_A[M,E(k)]$
      where $E(k)$ is the injective hull of $k=A/\m$.
      As we are assuming $A$ to contain $k$, $E(k) =\displaystyle{\lim_{\rightarrow}\, }
      \Hom_k[A/\m^\ell, k]$. Therefore for a finitely generated $A$-module $M$, we have:
      $$\Hom_A[M,E(k)] = \displaystyle{\lim_{\rightarrow}\, }\Hom_A[M, \Hom_k[A/\m^\ell, k]] = \displaystyle{\lim_{\rightarrow}\, }\Hom_k[M/\m^\ell M, k],$$
      thus our definition of Matlis duality  and that of Matlis
      are the same for a finitely generated $A$-module $M$.
    \end{remark}

    The Matlis duality discussed above also makes sense in the non-commutative setting of our paper: thus we have a local ring $(A,\m)$, an (associative) algebra $B$ containing $A$ in its center such that $B$ is finitely generated as an $A$-module. We will denote by $B^0$, the {\it opposite} algebra.
    If $M$ is a module over $B^0$, the Matlis dual of $M$, to be denoted as $M^\vee$, is the
$B$-module $$M^\vee=
\displaystyle{\lim_{\rightarrow}\, }\Hom_k[M/\m^\ell M, k].$$

As in the commutative case, $M\rightarrow M^\vee$
is an exact contravariant functor from the category
of finitely generated $B^0$-modules to category
of $B$-modules, taking finitely generated projective $B^0$-modules to injective $B$-modules,
finitely generated projective $B^0$-modules to artinian $B$-modules, and  for any $B$-module $M$, $M^\vee = \cup_{k\geq 1} \Ann(\m^k;M^\vee),$ where for any $B^0$-module $M$, we let $\Ann(\m^k;M)
= \{m \in M|\m^k m = 0\}$.

Further, $N\rightarrow N^\vee$ is a bijective correspondence between
finite artinian $B^0$-modules $N$ and finite artinian $B$-modules $N^\vee$. As before,
taking a projective resolution of a finite $B^0$-module, and applying the Matlis duality,
we have
the following proposition:

\begin{prop} \label{matlis2} Let $(A,\m)$ be a local 
  $k=A/\m$ algebra which is finitely generated over $k$.
      Let $B$ be an algebra containing $A$ in its center such that $B$ is finitely generated 
      as an $A$-module.
      Then any artinian $B$-module $N$ killed by a power of $\m$ has an injective resolution:
      $$0 \rightarrow N\rightarrow I_0 \rightarrow  I_1  \rightarrow  I_2 \rightarrow \cdots,$$
      by artinian and injective $B$-modules $I_j$ with $I_j= \cup_{k\geq 1} \Ann(\m^k;I_j)$ for all $j$. 
      \end{prop}

\section{Some generalities on Ext groups}
In this section,
$B$ will be an associative $\C$-algebra with unity containing in its center a finitely generated commutative algebra $A$ with the same unit, such that  $B$ is a finitely generated  $A$-module. The commutative  algebra $A$ over $\C$ comes with a  maximal ideal $\m$.
Let
$\widehat{A} = \displaystyle{ \lim_{\leftarrow}}(A/\m^n)$
be the  completion of $A$ at $\m$, and for any module $M$ over $A$, let $\widehat{M} = \displaystyle{ \lim_{\leftarrow}}(M/\m^nM)$; in particular, $\widehat{B} = \displaystyle{ \lim_{\leftarrow}}(B/\m^n B)$, and for any module $M$ over $B$ too, $\widehat{M} = \displaystyle{ \lim_{\leftarrow}}(M/\m^n M)$.

We begin with a proof of Schur's lemma, well-known in representation theory under a countability assumption (which is satisfied here).

\begin{lem} \label{schur}
A simple  $B$-module $M$  is finite dimensional over $\C$ on which $A$
operates by a central character $\omega:A\rightarrow \C$ whose kernel is a maximal ideal $\m$ in $A$.
 \end{lem}
\begin{proof} Being simple, $M$ is finitely generated over $B$, hence over $A$. Therefore we can apply Nakayama's lemma to conclude that there is a maximal ideal $\m$ in $A$ such that $M/\m M \not = 0$. But $A$ being central in $B$,
  $\m M$ is a $B$-submodule of $M$, hence by simplicity of $M$ as a
  $B$-module, $\m M=0$.
  \end{proof}

\begin{prop} \label{prop5.2} For any finitely generated modules $M,N$ over $B$,
  $$\Ext_B^i[N,M]
  \otimes_A \widehat{A} \cong \Ext_B^i[N,\wM]
  \cong \lim_{\leftarrow}\Ext_B^i[N,M/\m^nM].$$
  Further, if $\m$ acts as 0 on $N$, it also acts by 0 on $\Ext_B^i[N,M]$,
  and hence,$$\Ext_B^i[N,M] \cong \Ext_B^i[N,\wM]
  \cong \lim_{\leftarrow}\Ext_B^i[N,M/\m^nM].$$
  
\end{prop}

\begin{proof}
  Let $$\rightarrow P_n\rightarrow \cdots \rightarrow P_1\rightarrow P_0\rightarrow N\rightarrow 0,$$
  be a finitely generated
  projective resolution of $N$ as an $B$-module.

  By definition, for any $B$-module $M'$,
  $\Ext_B^i[N,M']$ is the cohomology of the cochain complex:
  $$0 \rightarrow \Hom_B[P_0, M']\rightarrow \Hom_B[P_1,M'] \rightarrow \cdots.$$
  In particular,
$\Ext_B^i[N,M/\m^kM]$ is the cohomology of the cochain complex:
  $$0 \rightarrow \Hom_B[P_0, M/\m^kM]\rightarrow
  \Hom_B[P_1,M/\m^kM]
  \rightarrow \cdots.$$

  Observe that:

  \begin{enumerate}

  \item Since $M$ and $P_i$ are finitely generated $B$-modules, and $B$ is noetherian, the modules $\Hom_B[P_i,M/\m^kM]$ appearing in the above complex are artinian $B$-modules,
    hence the above  projective system 
    of $B$-modules satisfies the Mittag-Leffler condition.
    Cohomology of such a complex commutes with inverse limits.

  \item By the definition of inverse limits, we have the equality
    $$\Hom_B[N',\wM]
    = \lim_{\leftarrow}\Hom_B[N',M/\m^nM],$$
    for any $B$-module $N'$.
  \end{enumerate}
  These two observations complete the proof of the
  assertion: $$\Ext_B^i[N,\wM] \cong \displaystyle{\lim_{\leftarrow}\,}\Ext_B^i[N,M/\m^nM].$$

  Since $\Hom_B[P,M]\otimes _A\widehat{A} \cong \Hom_B[P,\wM]$ (clearly true for finitely generated projective modules $P$ over $B$), and since $M\rightarrow \wM$ is an exact functor on the category of $B$-modules, one similarly proves that $
  \Ext_B^i[N,M] \otimes_A \widehat{A} \cong \Ext_B^i[N,\wM]$.

If   $\m$ acts by 0 on $N$, it also acts by 0 on $\Ext_B^i[N,M']$ for any $B$-module $M'$  (proved for example by observing that it is true
  for $\Hom_B[N,M']$ and then by the `usual' dimension shifting argument).

The proof of the  proposition is now complete.\end{proof}

\begin{prop} \label{prop5.3} Let $V$ be a  simple $B$-module with
  central character $\omega:A\rightarrow \C$ (assured by lemma \ref{schur})
  with $\m = {\rm ker}(\omega)$ the corresponding  maximal ideal in $A$.
Then for  any  finitely generated $B$-module $M$,
  $$\Ext_B^i[M,V]
  \cong \lim_{\rightarrow}
  \Ext_B^i[M/\m^nM,V].$$
\end{prop}

\begin{proof} Fix an injective resolution,
  $$0 \rightarrow V \rightarrow I_0 \rightarrow I_1 \rightarrow \cdots,$$
  such that for each $\ell \geq 0$, $I_\ell = \cup_{k\geq 1} \Ann(\m^k;I_\ell);$ this is assured by Proposition \ref{matlis2}. With this property for $I_\ell$,
  and since $M$ is finitely generated as a $B$-module, we  have:
  \begin{eqnarray*}
    \Hom_B[M,I_\ell] & = & \bigcup \Hom_B[M/\m^kM,I_\ell] \\
    & = & \lim_{\rightarrow}\Hom_B[M/\m^kM,I_\ell].
  \end{eqnarray*}

For any $B$-module $N$,  
  $\Ext_B^i[N, V]$ is, by definition, the cohomology of the cochain complex:
  $$0 \rightarrow \Hom_B[N, I_0]\rightarrow \Hom_B[N,I_1] \rightarrow \cdots.$$

  In particular,
$\Ext_B^i[M/\m^kM, V]$ is the cohomology of the cochain complex:
  $$0 \rightarrow \Hom_B[ M/\m^kM, I_0]\rightarrow \Hom_B[M/\m^kM, I_1] \rightarrow \cdots.$$
  As observed before,
  $      \Hom_B[M,I_\ell]  =  \displaystyle{\lim_{\rightarrow}\, }
  \Hom_B[M/\m^k M,I_\ell]$.
  Since cohomology of a cochain complex commutes with arbitrary direct limits, the proof of the proposition is complete.
\end{proof}

\section{An algebraic duality theorem}
We continue to assume that $B$ is an
associative $\C$-algebra with unity containing in its center a finitely generated commutative algebra $A$ with the same unit, such that  $B$ is a finitely generated  $A$-module.

\begin{prop}
  Let $V$ and $DV$ be two simple $B$-modules with the same central character $\omega:A\rightarrow \C$ with $\m $ the kernel of $\omega$.
  Assume that
  there is an integer $n=n(V)$ such that for any simple $B$-module $N$, the
  natural (cup product) pairing between finite dimensional vector spaces over $\C$:
  $$\Ext^i_B[V,N] \times \Ext^{n-i}_B[N,DV] \rightarrow \Ext^n_B[V,DV] \cong A/\m = \C,$$
  is perfect. Then for any finitely generated $B$-module $M$ also, $\Ext^i_B[V,M]$  and $\Ext^{i}_B[M,DV]$ are
  finite dimensional vector spaces over $\C$, and
  the natural (cup product) pairing between
  finite dimensional vector spaces over $\C$:
  $$\Ext^i_B[V,M] \times \Ext^{n-i}_B[M,DV] \rightarrow \Ext^n_B[V,DV] \cong \C,$$
  is perfect.
\end{prop}

\begin{proof} By generalities,
  $\Ext^i_B[V,M]$ and  $\Ext^{j}_B[M,DV]$
  are finitely generated $A$-modules on which
  $\m$ acts trivially
  (proved for example by observing that it is true
  for $\Hom_B[V,M]$ and  $\Hom_B[M,DV]$, and then by the `usual' dimension shifting argument).
  Hence, $\Ext^i_B[V,M]$ and  $\Ext^{j}_B[M,DV]$ are finite dimensional $\C$-vector spaces. Now we proceed in two steps.

  {\bf Step 1:} We will prove that if the natural pairing 
  $$\Ext^i_B[V,M] \times \Ext^{n-i}_B[M,DV]
  \rightarrow \Ext^n_B[V,DV] \cong \C,$$
  is perfect for $B$-modules $M_1$ and $M_2$, it is also the case for
  any module $M$ which is an extension of $M_2$ by $M_1$:
  $$0\rightarrow M_1 \rightarrow M \rightarrow M_2 \rightarrow 0.$$

  The proof of this follows from 5-lemma once we observe that the cup product appearing in the statement of the proposition has a naturality property
  under the boundary maps, call them $\delta_1^i:\Ext^{i-1}_B[V,M_2]
  \rightarrow  \Ext^{i}_B[V,M_1]$,
  and $\delta_2^{n-i}:\Ext^{n-i}_B[M_1,DV]
  \rightarrow \Ext^{n+1-i}_B[M_2,DV]$. Then
  for $\alpha \in \Ext^{i-1}_B[V,M_2]$ and $\beta \in \Ext^{n-i}_B[M_1,DV]$,
  
  $$\delta_1^i(\alpha) \cup \beta = \alpha \cup \delta^{n-i}(\beta) \in \Ext^n_B[V,DV] \cong \C.$$

  {\bf Step 2:} Given the conclusion in Step 1, for each integer $k\geq 1$, we have a perfect pairing,
  $$\Ext^i_B[V,M/\m^k M] \times \Ext^{n-i}_B[M/\m^k M,DV] \rightarrow \Ext^n_B[V,DV] \cong \C,$$

  By proposition \ref{prop5.2},
  $$\Ext^i_B[V,M] =  \lim_{\leftarrow} \Ext^{i}_B[V,M/\m^k M],$$
  and   by proposition \ref{prop5.3},
  $$\Ext^{n-i}_B[M,DV] = \lim_{\rightarrow} \Ext^{n-i}_B[M/\m^k M,DV].$$
  The conclusion of the proposition follows. \end{proof}

\section{Back to the duality theorem for $p$-adic groups}
Here is the main theorem of this paper  proved as a consequence of results from earlier sections.   Crucial use is of course being made of the theorem of Schneider-Stuhler which allows the conclusion of the theorem for $M$
  irreducible.

\begin{theorem} For $F$ a non-archimedean local field,
  let $G=\underline{G}(F)$ be the $F$-rational points of a reductive algebraic group $\underline{G}$. Let $V$ be an irreducible admissible representation of $G$, with Aubert-Zelevinsky involution $DV$. 
  Then for any finitely generated $G$-module $M$,
  the following is a perfect pairing of  finite dimensional vector spaces over $\C$:
  $$(7.1.1) \,\,\,\,\, \Ext^i_B[V,M] \times \Ext^{n-i}_B[M,DV] \rightarrow \Ext^n_B[V,DV] \cong \C.$$
  If $M$ is any smooth representation of $G$, then the pairing
  is nondegenerate in the sense that if $M = \displaystyle{\lim_{\rightarrow}\, } M_n$ of
finitely generated $G$-sub-modules $M_n$,
then $\Ext^{i}_G[V, M] =\displaystyle{\lim_{\rightarrow}\, } \Ext^{i}_G[V, M_n]$,
a direct limit of finite dimensional vector spaces over $\C$, and similarly,
$\Ext^{j}_G[M, DV]  = \displaystyle{\lim_{\leftarrow}\, } \Ext^{j}_G[M_n, DV]  $, an inverse limit
of finite dimensional vector spaces over $\C$, and the pairing  in $(7.1.1)$ is the direct limit of
perfect pairings on these finite dimensional spaces:
$$\Ext^{i}_G[V, M_n]  \times \Ext^{j}_G[M_n, DV]  \rightarrow 
\Ext^{i+j = d(\pi)}_G[V, DV]   \cong \C.$$
\end{theorem}
\begin{proof}
  Since $V$ is an irreducible admissible representation of $G$, $V^{G_n} \not = 0$ where $G_n$ is a sufficiently
  small congruence subgroup of $G$ (obtained say by using an embedding of $G=\underline{G}(F)$ inside $\GL_m(F)$,
  and intersecting with a congruence subgroup of $\GL_m(F)$).

  We now appeal to some generalities due to Bernstein, which goes in the theory of `Bernstein center',
  according to which smooth representations of $G$ which are generated by their $G_n$ fixed vectors is a
  direct summand of the category of all smooth representations of $G$, and this subcategory
  of smooth representations of $G$ is isomorphic to the category of $B$-modules for
  $B=\Hecke(G_n\backslash G/G_n)$. 
  Further, $B$ is a finite module over its center $A$ which is a finitely generated
  $\C$-algebra, cf. Corollary 3.4 in \cite{center}, as well as the notes of Bernstein \cite{lectures}.
  
 This  allows us to use theorems of the previous sections by identifying,
  $$\Ext^i_G[V,M] = \Ext^i_B[V_B,M_B],$$
  where $M_B$ (resp. $V_B$) is the submodule of $M$ (resp. $V$) consisting of $G_n$-fixed vectors which is a module for
  $B=\Hecke(G_n\backslash G/G_n)$.
  It is well-known that if $M$ is finitely generated as a $G$-module, then so is the $B$-module $M_B$.

The case of $M$ a general smooth representation of $G$ will be taken up in the next section.
\end{proof}

\begin{remark}
  It is a theorem of Aizenbud and Sayag in \cite{sayag}
  that for a $p$-adic group $G$, containing a subgroup $H$, if one knows that $\dim \Hom_H[\pi_1,\pi_2] < \infty$ for $\pi_1$ any irreducible admissible representation of $G$, and $\pi_2$ of $H$, and if $[G \times H]/\Delta H$ is a
  spherical variety, then for any compact open subgroup $K$ of $H$, and any irreducible representation $\pi_1$ of $G$, $\pi_1^K$ is a finitely generated module
  over $\Hecke(K\backslash H/K)$. Thus taking $G$ to be $\GL(n+1),\U(n+1),\SO(n+1)$
  and $H$ to be the corresponding subgroup $\GL(n),\U(n),\SO(n)$, we get a rich supply of finitely generated representations of Hecke algebras that Aizenbud and Sayag  call {\it locally finitely generated}.
  \end{remark}
\section{General smooth representations}

Since the pairing
  $$\Ext^i_B[V,M] \times \Ext^{n-i}_B[M,DV] \rightarrow \Ext^n_B[V,DV] \cong \C,$$
has a functorial structure for $B$-modules $M$, to understand it for
an arbitrary $B$-module $M$,  the following
two simple lemmas suffice.

\begin{lem} Let $V$ be a finitely generated $B$-module, and $M$ a general $B$-module.
  Write $M = \displaystyle{\lim_{\rightarrow}M_n}$ of finitely generated $B$-submodules.
  Then we have,
    $$\Ext_B^i[V,M] =
  \displaystyle{\lim_{\rightarrow}} \Ext_B^i[V,M_n] .$$
\end{lem}
\begin{proof}
  Let $$\rightarrow P_n\rightarrow \cdots \rightarrow P_1\rightarrow P_0\rightarrow
  V\rightarrow 0,$$
  be a projective resolution of $V$ as a $B$-module consisting of finitely generated $B$-modules.

  By definition,
  $\Ext_B^i[V,M]$ is the cohomology of the cochain complex:
  $$0 \rightarrow \Hom_B[P_0, M]\rightarrow \Hom_B[P_1,M] \rightarrow \cdots.$$
  Since $P_i$ are finitely generated as  $B$-modules,
  $$\Hom_B[P_i,M] =
\displaystyle{\lim_{\rightarrow}} \Hom_B[P_i,M_n] .$$
The Lemma  now follows on  noting that cohomology commutes with direct limits.
  \end{proof}
  
Analogously, we have:

\begin{lem} Let $V$ be an irreducible $B$-module,
  and $M$ a general $B$-module.
  Write $M = \displaystyle{\lim_{\rightarrow}M_n}$ of finitely generated $B$-submodules.
  Then we have, $$\Ext_B^i[M,V] =
  \displaystyle{\lim_{\leftarrow}} \Ext_B^i[M_n,V].$$
  
\end{lem}
\begin{proof}
  Let $$0 \rightarrow V \rightarrow I_0 \rightarrow I_1\rightarrow I_2\rightarrow \cdots $$
  be an injective resolution of $V$ as a $B$-module. 
  Since $V$ is an irreducible $B$-module hence finite dimensional by lemma \ref{schur}, one can assume that
  $I_j$ are artinian as $B$-modules as provided by proposition \ref{matlis2}.

  By definition,
  $\Ext_B^i[M,V]$ is the cohomology of the cochain complex:
  $$0 \rightarrow \Hom_B[M,I_0]\rightarrow \Hom_B[M,I_1]
  \rightarrow \cdots.$$

  Since $I_j$ are artinian, for any finitely generated submodule $M_n \subset M$, $\Hom_B[M_n,I_j]$ are artinian $A$ modules, therefore the Mittag-Leffler condition holds good for the projective system of cochain complexes:
  $$0 \rightarrow \Hom[M_n, I_0] \rightarrow \Hom[M_n,I_1] \rightarrow \Hom[M,I_2]\rightarrow \cdots $$

  Since by definition, $\Hom_B[M,I_i] =\displaystyle{\lim_{\leftarrow}} \Hom_B[M_n,I_i]$, the proof of the lemma is completed by the generality that cohomology commutes with inverse limits
  when Mittag-Leffler condition is satisfied.
  \end{proof}

\section{An application of the  duality theorem}
In this section  we give a sample application of
the Schneider-Stuhler duality theorem to branching laws.

The paper \cite{EP} suggests that branching problems (say from $\SO_{n+1}(F)$ to $\SO_n(F)$) which have such a simple eventual answer for
$\dim \Hom_{\SO_n(F)}[\pi_1,\pi_2]$ where
$\pi_1$ is an irreducible admissible representation of $\SO_{n+1}(F)$ and
$\pi_2$ is an irreducible admissible representation of 
$\SO_n(F)$ (and assume for instance that they are both tempered, or more generally belong to generic Vogan packets) is because higher Ext's are zero:
$$\Ext^i_{\SO_n(F)}[\pi_1,\pi_2] =0 {\rm ~~~for~~~} i>0.$$

The Schneider-Stuhler duality theorem then allows one to answer a natural question: what are the $\SO_n(F)$-submodules of an irreducible admissible module $\pi_1$ of $\SO_{n+1}(F)$ when
restricted to $\SO_n(F)$?
The answer is that typically none (except of course supercuspidals), but that
there are submodules not in the sense that 
$\Hom_{\SO_n(F)}[\pi_2,\pi_1] \not = 0$ but in the sense that 
$\Ext^{d(\pi_2)}_{\SO_n(F)}[\pi_2,\pi_1] \not = 0$.

As a very modest application of Schneider-Stuhler theorem, we  prove the 
following proposition giving a complete classification
of irreducible submodules $\pi$ of the tensor product $\pi_1 \otimes \pi_2$ of two representations $\pi_1,\pi_2$ of $\GL_2(F)$ with the product of their central characters trivial. 
As this proposition shows, it is rare for a non-supercuspidal representation of a subgroup
$H$ to appear as a subrepresentation of a representation of a group $G$ when restricted to $H$
(in this case from $\GL_2(F) \times \GL_2(F) $ to the diagonal $\GL_2(F)$); 
this is to be contrasted with their abundant appearance as a quotient studied in \cite{prasad1}. In fact, it should be considered somewhat of a
surprising conclusion that there can be  a non-supercuspidal submodule at all!

\begin{prop}
Let $\pi_1, \pi_2$ be two irreducible admissible infinite dimensional 
representations of $\GL_2(F)$ with product of their
central characters trivial. Then the following is a complete list of irreducible sub-representations
$\pi$ of $\pi_1 \otimes \pi_2$ as $\PGL_2(F)$-modules.

\begin{enumerate}
\item $\pi$ is a supercuspidal representation of $\PGL_2(F)$, and appears as a quotient of $\pi_1 \otimes \pi_2$.

\item $\pi$ is a twist of the Steinberg representation, which we assume by absorbing the twist in 
$\pi_1$ or $\pi_2$ to be the Steinberg representation $\St$ of $\PGL_2(F)$. Then $\St$ is a 
submodule of $\pi_1 \times \pi_2$ if and only if 
$\pi_1,  \pi_2$ are both irreducible
principal series representations, and $\pi_1 \cong \pi_2^\vee$.
\end{enumerate}
\end{prop}

\begin{proof} Since a supercuspidal representation of $\PGL_2(F)$ is a projective module, it appears as a submodule of $\pi_1 \otimes \pi_2$ if and only if it appears as a quotient module, thus the assertion of the proposition for $\pi$ supercuspidal is clear.

If the representation $\pi$ of $\PGL_2(F)$ is not supercuspidal, then $d(\pi)=1$, and
there is a non-degenerate pairing:
$$\Hom_{\PGL_2(F)}[\pi, \pi_1 \otimes \pi_2]
\times \Ext^{1}_{\PGL_2(F)}[\pi_1 \otimes \pi_2, D(\pi)]  
\rightarrow 
\Ext^{1}_{\PGL_2(F)}[\pi, D(\pi)]   \cong \C, $$
allowing one to calculate $\Hom_{\PGL_2(F)}[\pi, \pi_1 \otimes \pi_2]$
in terms of $\Ext^{1}_{\PGL_2(F)}[\pi_1 \otimes \pi_2, D(\pi)] $ which
in turn is calculated in terms of Euler-Poincare function (replacing $D\pi$ by $\pi$ for notational ease):
\begin{eqnarray*}
  \EP^1_{\PGL_2(F)}[\pi_1 \otimes \pi_2, \pi] & = &
  \dim\Hom_{\PGL_2(F)}[\pi_1 \otimes \pi_2, \pi]
  - \dim \Ext^1_{\PGL_2(F)}[\pi_1 \otimes \pi_2, \pi],
\end{eqnarray*}
which is relatively straightforward to calculate --- as we briefly indicate in the next paragraph --- and then using knowledge about
$\dim\Hom_{\PGL_2(F)}[\pi_1 \otimes \pi_2, \pi]$ from \cite{prasad1} to calculate
$ \dim \Ext^1_{\PGL_2(F)}[\pi_1 \otimes \pi_2, \pi].$

For example, if $\pi$ is an irreducible principal series representation of $\PGL_2(F)$, then $D(\pi) = \pi$,   and in this case it can be seen that 
$$\Ext^1_{\PGL_2(F)}[\pi_1 \otimes \pi_2, \pi] = \Hom_{\PGL_2(F)}[\pi, \pi_1 \otimes \pi_2] = 0.$$
The assertion  on vanishing of $\Ext^1$, one of the main conjectures in \cite{EP} in some generality, is easy to prove if one of $\pi_i$ is supercuspidal; if neither of $\pi_i$ is supercuspidal, then $\pi_1\otimes \pi_2$
is --- by Mackey orbit theory --- 
equal to $\ind_T^{\PGL_2(F)}(\chi)$ (up to an admissible module),
where $\chi:T\rightarrow \C^\times$ is a character on a maximal split torus $T$, and the statement on Euler-Poincare and hence $\Ext^1$ 
on $\PGL_2(F)$ reduces by Frobenius reciprocity to an assertion on the split torus.

We next turn to $\pi = \Sp$ in which case we have $D(\Sp) = \C$. By the 
Schneider-Stuhler theorem, calculation of $\Hom_{\PGL_2(F)}[\Sp, \pi_1 \otimes \pi_2]$ is reduced to that of 
$\Ext^{1}_{\PGL_2(F)}[\pi_1 \otimes \pi_2, \C]$ where $\pi_1$ and $\pi_2$ are two irreducible admissible representations of $\GL_2(F)$ with central characters $\chi$ and $\chi^{-1}$. For this, note the generality:
$$\Ext^1_{\PGL_2(F)}[\pi_1\otimes \pi_2, \C] \cong
\Ext^1_{\GL_2(F), \chi}[\pi_1, \pi_2^\vee] .$$

But by another application of Schneider-Stuhler theorem,
$$
\dim \Ext^1_{\GL_2(F), \chi}[\pi_1, \pi_2^\vee] = \dim \Hom_{\GL_2(F)}[\pi_2^\vee, \pi_1],$$
if $\pi_2^\vee = D(\pi_2^\vee)$ is an irreducible principal series representation of $\GL_2(F)$. Thus,
$$\Hom_{\PGL_2(F)}[\Sp, \pi_1 \otimes \pi_2] = \C,$$
if and only if $\pi_1 \cong \pi_2^\vee$ are irreducible principal series representations of $\GL_2(F)$. The cases of the proposition
we may appear to have missed out when one of $\pi_i$ is the twist
of the Steinberg representation, or when $\pi$ is the trivial representation follows even more easily which we leave to the reader.\end{proof}

An explicit embedding of $\Sp$ into $\ind_T^{\PGL_2(F)}\C$ was constructed in Lemma 5.4 of \cite{prasad1}, which allows one to
get an embedding of $\Sp$ inside $\pi_1\otimes\pi_1^\vee$ for $\pi_1$ any principal series representation of $\GL_2(F)$.

\end{document}